\documentclass[12pt]{amsart}
\usepackage[all]{xy}
\usepackage[parfill]{parskip}

\usepackage{color}

\usepackage{amsmath, amscd, graphicx, latexsym,  times}
\usepackage{color,graphicx}
\usepackage[abs]{overpic}
\oddsidemargin .25in \theoremstyle{plain}
\newtheorem{dummy}{anything}[section]
\newtheorem{theorem}[dummy]{Theorem}

\newtheorem{lemma}[dummy]{Lemma}

\newtheorem{proposition}[dummy]{Proposition}

\newtheorem{corollary}[dummy]{Corollary}

\theoremstyle{definition}

\newtheorem{remark}[dummy]{Remark}

\begin{document}

\title{Fillings of unit cotangent bundles of nonorientable surfaces}

\author{Youlin Li and Burak Ozbagci}

\address{School of Mathematical Sciences, Shanghai Jiao Tong University, Shanghai 200240, China}
\email{liyoulin@sjtu.edu.cn}
\address{Department of Mathematics, UCLA, Los Angeles, CA 90095 and 
Department of Mathematics, Ko\c{c} University, 34450,  Istanbul,
Turkey}
\email{bozbagci@ku.edu.tr}

\subjclass[2000]{}

\thanks{The first author is partially supported by grant no. 11471212 of the National Natural Science Foundation of China.}


\begin{abstract}

We prove that any minimal weak symplectic filling of the canonical contact structure on the unit cotangent bundle of a nonorientable closed surface other than the  real projective plane is s-cobordant rel boundary to the disk cotangent bundle of the surface.  If the nonorientable surface is the Klein bottle, then we show 
that the minimal weak symplectic filling is unique up to homeomorphism.   
\end{abstract}

\maketitle

\section{Introduction}\label{sec: intro}

  Let $M$ denote a closed surface which is not assumed to be orientable. The bundle of cooriented 
lines tangent to $M$ has a projection onto $M$, which we denote by $\pi$. For a  point $q$ in $M$ and 
a cooriented line $u$ in $T_q M$,
we denote by
$\xi_{(q,u)}$ the cooriented
  plane  described uniquely by the
  equation $\pi_*(\xi_{(q,u)}) = u \in T_q M$. The canonical contact structure $\xi_{can} $ on the bundle of cooriented lines
  tangent to $M$ consists of these planes (see, for example, \cite{ma}).

    The bundle of cooriented lines
  tangent to $M$ can be identified with the unit cotangent
        bundle $ST^*M$, once  $M$ is equipped with a Riemannian
        metric. Under this identification, the contact structure $\xi_{can}$ is given by the kernel of
        the Liouville $1$-form $\lambda_{can}$.  It follows that the contact $3$-manifold $(ST^*M, \xi_{can}) $ is Stein fillable, with one filling given by the disk
        cotangent bundle $DT^*M$.

       Let  $\Sigma_g$ denote the closed orientable surface of genus $g \geq 0$.  The  unit cotangent bundle $ST^*\Sigma_0$ is diffeomorphic to the real
        projective space $\mathbb{RP}^3$, and $\xi_{can}$ is the  unique tight contact
        structure in $\mathbb{RP}^3$, up to isotopy (cf. \cite{h}). McDuff
        \cite{mc} showed that any minimal symplectic filling of
        $(\mathbb{RP}^3, \xi_{can})$ is diffeomorphic to $DT^*\Sigma_0$.  

The unit cotangent bundle $ST^*\Sigma_1$ is diffeomorphic to the
        $3$-torus $T^3$ and Eliashberg \cite{el} showed that $\xi_{can}$
        is the  unique strongly symplectically  fillable contact structure in $T^3$, up to contactomorphism. 
        Moreover, Stipsicz {\cite{s} proved that any Stein filling of $(T^3, \xi_{can})$  is homeomorphic to the disk cotangent bundle $DT^*\Sigma_1 \cong T^2 \times D^2$. This result was improved by  Wendl \cite{we}, who showed that, in fact,  any minimal strong symplectic
        filling of $(T^3, \xi_{can})$ is symplectic deformation equivalent  to 
          $DT^*\Sigma_1$ equipped with its canonical symplectic structure. 

Recently, Sivek and Van Horn-Morris \cite{sm} proved that, for $g \geq 2$, any Stein filling 
of the contact $3$-manifold $(ST^*\Sigma_g, \xi_{can})$  is  s-cobordant rel boundary to the disk
cotangent bundle $DT^*\Sigma_g$. 

Moreover, Li, Mak and Yasui proved that, for $g \geq 2$, $(ST^*\Sigma_g, \xi_{can})$ admits minimal strong symplectic fillings with arbitrarily large $b^+_2$ (see \cite[Proof of Corollary 1.6]{lmy})) despite the fact that 
any exact filling of $(ST^*\Sigma_g, \xi_{can})$ has the same integral homology and intersection form as  $DT^*\Sigma_g$ \cite[Theorem 1.4]{lmy}. 

In this paper, we study the topology of the symplectic fillings of the canonical contact structure $\xi_{can}$ on the unit cotangent bundle of any \emph{nonorientable} closed  surface. A significant feature in the nonorientable surface case is that $\xi_{can}$ on the unit cotangent bundle is supported by a planar open book decomposition \cite{oo}. Therefore, the topology of the symplectic fillings is greatly restricted by the results in \cite{et}, \cite{nw}, \cite{wa}, and \cite{we}. Most importantly, according to a theorem of  
Niederkr\"{u}ger and Wendl \cite{nw},  any weak symplectic filling of $\xi_{can}$ is symplectic deformation equivalent to a blow up of one of its Stein fillings.

Here we show that the canonical contact structure on the unit cotangent bundle of any nonorientable closed surface other than the real projective plane admits a unique minimal weak symplectic filling, up to s-cobordism rel boundary. More precisely, we prove the following.

\begin{theorem} \label{thm: main} Let $N_k$ to denote the nonorientable closed surface
       obtained by the connected sum of $k \geq 1$ copies of the real
       projective plane $\mathbb{RP}^2$. Then, for $k \geq 2$, any minimal weak symplectic filling of the canonical contact structure $\xi_{can}$ on the unit cotangent bundle 
$ST^*N_k$ is s-cobordant rel boundary to the disk cotangent bundle $DT^*N_k$.  
\end{theorem}

Suppose that $X_0$ and $X_1$ are compact $4$-manifolds such that $\partial X_0$ is diffeomorphic 
to $\partial X_1$. Then $X_0$ and $X_1$ are said to be \emph{s-cobordant rel boundary} (cf. \cite[page 89]{fq}) if there exists a 
 compact $5$-manifold $Z$ whose boundary $\partial Z$ is diffeomorphic to the  union of $X_0$, $X_1$ and $\partial X_0 \times [0,1]$ such that for each $i=0, 1$, the inclusion $X_{i} \rightarrow Z$   is a simple homotopy equivalence.  It follows that 
if  $X_0$ and $X_1$ are s-cobordant rel boundary,  then they are simple homotopy equivalent. The reader can turn to \cite{c} for more on simple homotopy equivalences.  If $\pi_1(Z)$ is polycyclic, then according to \cite[7.1A Theorem]{fq},  $Z$ is homeomorphic to $X_0 \times [0,1]$, and in particular $X_0$ is homeomorphic to $X_1$.  

Since the fundamental group of the Klein bottle is polycyclic, we have the following corollary. 

\begin{corollary}
Any minimal weak symplectic filling of the canonical contact structure $\xi_{can}$ on the unit cotangent bundle $ST^*N_2$ of the Klein bottle $N_2$ is homeomorphic to the disk cotangent bundle $DT^*N_2$.
\end{corollary}

The unit cotangent bundle $ST^*N_1$ of the real projective plane $N_1=\mathbb{RP}^2$ is diffeomorphic to the
        lens space  $L(4,1)$ and $\xi_{can}$
        is the  unique universally tight contact structure in $L(4,1)$, up to
        contactomorphism. McDuff
        \cite{mc} showed that $(L(4,1), \xi_{can})$ has two minimal symplectic fillings
        up to diffeomorphism: 

(i) The disk cotangent bundle $DT^*N_1$, which is a rational homology $4$-ball,
        and 

(ii) The disk bundle over the sphere with Euler number $-4$.

Both of these fillings are in fact Stein and clearly not homotopy equivalent (and hence not s-cobordant), since the latter is simply-connected while the former is not.

\subsection{Notation} If $\alpha$ is a simple closed curve on an oriented surface $\Sigma$, we denote the positive (a.k.a. right-handed) Dehn twist along $\alpha$ by $D(\alpha)$. We use $Map(\Sigma, \partial \Sigma)$ for the mapping class group of isotopy classes of 
diffeomorphism of the surface $\Sigma$  which are assumed to be equal to the identity at the boundary $\partial \Sigma$.  We use functional notation for the products in  $Map(\Sigma, \partial \Sigma)$, i.e., $D(\beta)D(\alpha)$ means that we first apply 
$D(\alpha)$. 

If $(X, \omega)$ is a symplectic $4$-manifold and we are only interested in its diffeomorphism type, then we suppress  
$\omega$ from the notation. Similarly, for a Stein surface $(W, J)$, we suppress the complex structure $J$ from the notation if it is irrelevant for the discussion.  The reader is advised to turn to \cite{ozst} for the background material that we will use  throughout the paper.

\section{Homology of the fillings} \label{sec: homology} Our goal in this section is to prove the 
following proposition. 

\begin{proposition} \label{prop: hom} Suppose that  
$W$ is a minimal weak symplectic filling of $(ST^{*}N_{k}, \xi_{can})$.  Then $$H_1(W; \mathbb{Z}) = H_1(DT^*N_k ; \mathbb{Z})=\mathbb{Z}^{k-1} 
\oplus \mathbb{Z}_2, $$ and  $$ H_2(W; \mathbb{Z}) = H_2(DT^*N_k ; \mathbb{Z})=0$$ provided that $k \geq 2$. 
\end{proposition} 

\begin{remark} Note that Proposition~\ref{prop: hom} does not hold for $k=1$,  since the contact $3$-manifold $(ST^{*}N_{1}=L(4,1), \xi_{can})$ has a \emph{simply-connected} Stein filling, namely the disk bundle over the sphere with Euler number $-4$, whose second homology group is $\mathbb{Z}$.\end{remark}

We rely on the following results to prove Proposition~\ref{prop: hom}.

\begin{lemma}[\cite{oo}] \label{lem: openbook}
Let $V_{0}, V_{1}, \dots, V_{k}, V_{k+1}$ be the simple closed curves shown in Figure \ref{fig: monod} on the planar 
surface $F_{k}$ with $2k+2$ boundary components, and let
$$\phi_{k}:= D(V_{0}) D(V_{1})  \cdots D(V_{k}) D(V_{k+1}) \in Map(F_k, \partial F_k).$$
Then, for all $k \geq 1$, the open book $(F_k,
\phi_{k})$ is adapted to $(ST^{*}N_{k}, \xi_{can})$.
\end{lemma}

\begin{figure}[htb]
\begin{overpic}
{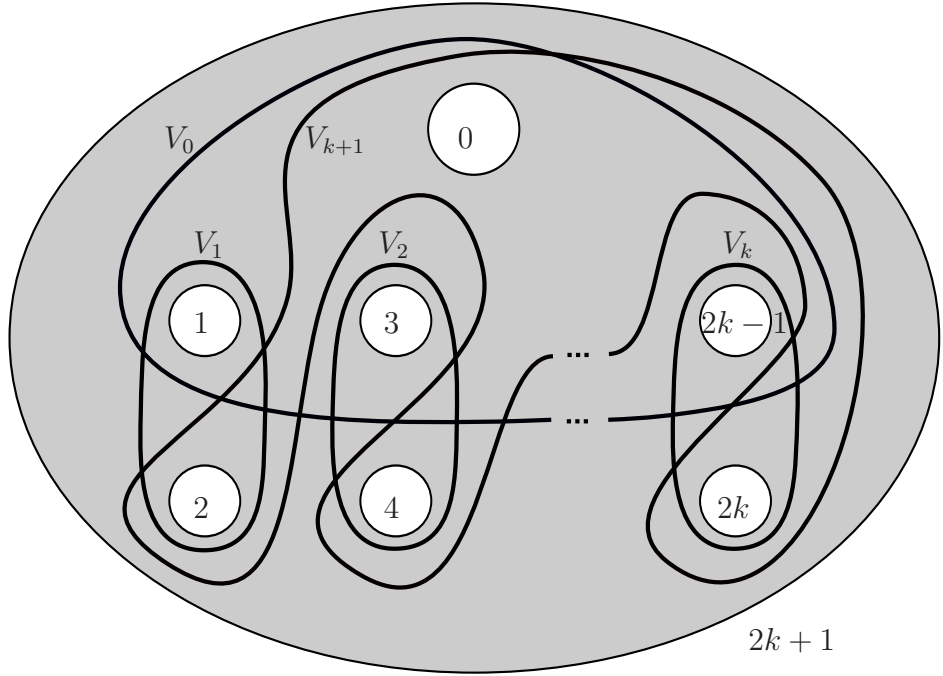}
\put(170, 200){$0$}
\put(70, 130){$1$}
\put(70, 60){$2$}
\put(142, 130){$3$}
\put(142, 60){$4$}
\put(262, 130){$2k-1$}
\put(268, 60){$2k$}
\put(280, 10){$2k+1$}
\put(59, 200){$V_{0}$}
\put(70, 160){$V_{1}$}
\put(140, 160){$V_{2}$}
\put(270, 160){$V_{k}$}
\put(112, 200){$V_{k+1}$}

\end{overpic}
\caption{The curves $V_{0}, V_{1}, \dots, V_{k}, V_{k+1}$ on the planar surface  $F_{k}$. Each boundary component of $F_k$ , denoted by 
$c_j$ in the text, for  $j=0, 1,  2, \ldots, 2k+1$, is labeled by $j$ in the figure. }
\label{fig: monod}
\end{figure}

\begin{remark} \label{rem: kir} For any $k \geq 1$, the disk cotangent bundle  $DT^*N_k$ is a Weinstein (and hence Stein) filling \cite[Example 11.12 (2)]{ce} of its boundary $(ST^{*}N_{k}, \xi_{can})$. To see this with another point of view, one can directly check that the total space of the Lefschetz fibration over the disk whose boundary has the induced open book decomposition $(F_k,
\phi_{k})$ given in Lemma~\ref{lem: openbook} is diffeomorphic to the disk cotangent bundle $DT^*N_k$ (cf. \cite[Appendix]{oo}). \end{remark} 

\begin{theorem}[Niederkr\"{u}ger and Wendl \cite{nw}] \label{thm: nw}
Every weak symplectic filling  $(X, \omega)$ of a contact 3-manifold $(Y, \xi)$ supported by
a planar open book decomposition, is symplectic deformation
equivalent to a blow-up of a Stein filling of $(Y, \xi)$.
\end{theorem}

\begin{remark} \label{rem: wen} Note that Wendl  \cite{we}  had previously proved Theorem~\ref{thm: nw} for 
strong symplectic fillings.\end{remark}

In particular, in order to classify minimal weak symplectic  fillings of $(Y, \xi)$ supported by a planar open book decomposition, it suffices to study positive factorizations of the given monodromy of that open book decomposition. Note that every Stein filling admits an allowable Lefschetz fibration compatible with the given open book on the boundary (cf. \cite{lp}, \cite{ao}).

In the following, we will apply Theorem~\ref{thm: nw}  to the monodromy of the planar 
open book decomposition explicitly described in Lemma~\ref{lem: openbook}.  

\begin{lemma}\label{lem: monocurve}
Assume that $k\geq3$. Then any positive factorization of $\phi_k$ in Lemma~\ref{lem: openbook} consists of Dehn twists 
$D(V'_{0}),  D(V'_{1}),  \ldots,  D(V'_{k+1})$, where each curve $V'_{j}$ encloses the same holes as $V_j$, for $j =0, 1, 
 \ldots ,  k+1$. 
\end{lemma}

\begin{proof} In order to study positive factorizations of the given monodromy of a planar open book, we make use of a technique due to Plamenevskaya and Van Horn-Morris \cite{pv} (see also \cite{ka}, \cite{kl}). Suppose that the monodromy $\phi_k$ in Lemma~\ref{lem: openbook}  has a factorization into positive Dehn twists along some curves. In the following we will refer to any such curve as a \emph{monodromy} curve. 

Recall that $c_j$'s are the boundary components  of the planar surface $F_k$ depicted in Figure~\ref{fig: monod}. For each $j\in\{0, 1, 2,$ $\ldots, 2k\}$, the  \emph{multiplicity} $M_j (\phi_k)$ is defined as the number of monodromy curves enclosing $c_j$,   and similarly, for each $i \neq j \in\{0, 1, 2,$ $\ldots, 2k\}$, the \emph{joint multiplicity} $M_{i, j} (\phi_k)$ is defined as the number of monodromy curves enclosing $c_{i}$ and $c_{j}$, in any positive factorization of $\phi_k$. The point is that these multipilicities are independent of the positive factorizations of $\phi_k$ (see  \cite{ka}, \cite{kl}).

It is easy to compute the multiplicities and joint multiplicities of $\phi_k$ using the positive factorization given in Lemma~\ref{lem: openbook}.

\begin{enumerate}
\item For any $j\in\{0, 1, 2,$ $\ldots, 2k\}$, $M_j=2$. 
\item For any $j\in\{1,  2,\ldots, 2k\}$, $M_{0, j}=1$. 
\item For any $h\in\{1, 3,$ $ \ldots, 2k-1\}$ and $l\in\{2, 4,$  $\ldots, 2k\}$, $M_{h,l}=0$ if $l\neq h+1$. 
\item For any $h\in\{1, 3,$ $ \ldots, 2k-1\}$, $M_{h,h+1}=1$. 
\item For any $h_1, h_2\in\{1, 3,$ $\ldots, 2k-1\}$, $M_{h_1, h_2}=1$.  
\item For any $l_1, l_2\in\{2, 4,$ $\ldots, 2k\}$, $M_{l_1, l_2}=1$.
\end{enumerate}

 We first claim that there is no boundary parallel monodromy curve in any positive factorization of $\phi_k$, a proof of which is spelled out in the next four paragraphs.

Suppose that there is a monodromy curve which is parallel to $c_0$. Then since $M_0=2$, by $(1)$,  there is another monodromy curve which encloses $c_0$ as well.  Since, by $(2)$, $M_{0, j}=1$ for any $j\in\{1, 2,  \ldots, 2k\}$, the second curve enclosing $c_0$ must enclose $c_j$ for all $j\in\{1, 2,  \ldots, 2k\}$. This contradicts to $(3)$ by taking $h=1$ and $l=4$.

Suppose that there is a monodromy curve which is parallel to $c_i$, for $i \in \{1, 3,$ $ \ldots, 2k-1\}$. Then there is another monodromy curve enclosing $c_i$, since by $(1)$,  $M_i=2$. But the second curve enclosing $c_i$ must enclose $c_{i+1}$ 
by $(4)$  and also  $c_j$ for all $j \in \{0, 1, 3,$ $ \ldots, 2k-1\}$, by $(2)$ and $(5)$. This contradicts to $(3)$, by taking $h=i$ and $l \in\{2, 4,$  $\ldots, 2k\} \setminus \{i+1\}$. 

Suppose that there is a monodromy curve which is parallel to $c_i$, for $i  \in\{2, 4,$ $\ldots, 2k\}$. Then there is another monodromy curve enclosing $c_i$, since by $(1)$,  $M_i=2$. But the second curve enclosing $c_i$ must enclose $c_{i-1}$ 
by $(4)$  and also  $c_j$ for all $j \in \{0, 2, 4,$ $\ldots, 2k\}$, by $(2)$ and $(6)$. This contradicts to $(3)$, by taking $l=i$ and $h \in \{1, 3,$  $\ldots, 2k-1\} \setminus \{i-1\}$. 

Finally, we observe that by $(3)$,  there is no monodromy curve which is parallel to the outer boundary component $c_{2k+1}$. We only need to assume that $k \geq 2$ so far in the proof, since $(3)$ is vacuous for $k=1$ (see Remark~\ref{rem: k2} (1)).

Next we note that, by $(4)$,  for any $i\in\{1, 3,$ $ \ldots, 2k-1\}$, there must be a monodromy curve in the factorization enclosing $c_i$ and $c_{i+1}$. We claim that this curve cannot enclose any other boundary components.  This monodromy curve cannot enclose $c_j$ for any $j\in\{1, 2,  \ldots, 2k\}\setminus\{i, i+1\}$, since by $(3)$,  $M_{i,j}=0$ for all $j 
 \in\{2, 4,$ $\ldots, 2k\} \setminus \{i+1\}$ and  $M_{i+1,j}=0$  for all $j 
 \in\{1, 3,$ $\ldots, 2k-1\} \setminus \{i\}$. Suppose that $c_0$ is enclosed by this monodromy curve. Then, by $(1)$ and $(2)$,  there must be  another monodromy curve enclosing $c_0$ and $c_j$ for all $j\in\{1,2,\ldots, 2k\}\setminus \{i, i+1\}$. 
This contradicts to $(3)$, provided that $k \geq 3$. For each $i\in\{1, 3,$ $ \ldots, 2k-1\}$, we denote by $V'_h$, the monodromy curve  enclosing only $c_i$ and $c_{i+1}$,    where $h = (i+1)/2$.  Note that, for $k=2$, there is another possibility of configuration of monodromy curves, as explained in Remark~\ref{rem: k2}  (2) below. 

By $(1)$ and $(2)$, there must be two more monodromy curves in the factorization, in addition to the ones described in the previous paragraph, both enclosing $c_0$. We describe these in the next two paragraphs. 

One of these, which we denote by $V'_0$,  must enclose $c_0$ and $c_1$ by $(2)$. We claim that $V'_0$ encloses $c_j$ if and only if  $j\in\{0, 1, 3, $  $\ldots, 2k-1\}$. By $(3)$, $V'_0$ cannot enclose $c_l$ for $l\in\{4, 6, $  $\ldots, 2k\}$. If we 
assume that it encloses $c_2$, then  it cannot enclose $c_j$ for  any $j \in \{3,5, $  $\ldots, 2k-1 \}$ by $(3)$, and therefore,   by $(1)$ and $(2)$, there is another monodromy curve enclosing $c_0$ and $c_j$ for all $j \in \{3,4,5, \ldots , 2k-1, 2k\}$, which again contradicts to $(3)$, provided that $k\geq 3$.  Suppose now that $V'_0$ does not enclose $c_h$ for some $h\in\{3, 5,$ $\ldots, 2k-1\}$. Then there must be another monodromy curve enclosing $c_0$, $c_h$ and $c_l$ for all $l\in \{2, 4,$  $\ldots, 2k\}$, by $(1)$ and $(2)$.  This contradicts to $(3)$ for $l \neq h+1$. 

The last monodromy curve,  which we denote by $V'_{k+1}$,   must enclose $c_0$ and $c_2$. By an argument similar to the above paragraph, $V'_{k+1}$ 
encloses $c_j$ if and only if  $j\in\{0, 2, 4,$  $\ldots, 2k\}$,   provided that $k\geq 3$.
\end{proof}

\begin{remark} \label{rem: k2} 
\begin{enumerate}
\item In order to rule out the existence of 
boundary  parallel monodromy curves in the factorization, we used $(3)$, which is vacuous for $k=1$.   As a matter of fact, $\phi_1$ has two positive factorizations,  $$D(V_{0}) D(V_{1}) D(V_{2}) =  D(c_{0}) D(c_{1}) D(c_{2}) D(c_{3}), $$ the latter consisting of only boundary parallel curves. This equality in the mapping class group is the well-known lantern relation. 
\item For $k=2$, there is another possibility of configuration of monodromy curves which we can not rule out by the above argument. Namely, the four monodromy curves in the factorization may enclose the following set of boundary curves $$\{c_0,c_1,c_2\}, \{c_0,c_3,c_4\}, \{c_1,c_3\}, \{c_2,c_4\},$$ respectively.  
\end{enumerate}

\end{remark}

\begin{proof}[Proof of Proposition~\ref{prop: hom}]
Each positive factorization of the monodromy $\phi_k$ in Lemma~\ref{lem: openbook} yields a Lefschetz fibration over the disk whose boundary has the induced open book decomposition $(F_k, \phi_k)$.  The regular fiber of this  Lefschetz fibration is $F_k$ and the vanishing cycles are exactly the monodromy curves in the given factorization. Therefore, one obtains a handlebody decomposition and the corresponding Kirby diagram of the total space of the fibration as follows. 

First of all, a neighborhood of the regular fiber $F_k$ is diffeomorphic to  $F_k \times D^2$ which is obtained by attaching $2k+1$ $1$-handles to the unique $0$-handle. In the corresponding Kirby diagram, one can visualize the fiber $F_k$  as depicted  in Figure~\ref{fig: monod},  whose orientation is induced from the standard orientation of $\mathbb{R}^2$. Note that for each $j\in\{0, 1, 2,\ldots, 2k\}$, the corresponding $1$-handle can be conveniently depicted as a dotted circle passing through the hole labeled by  $j$, and linking once the outer boundary component labeled by $2k+1$.  In addition,  a $2$-handle with framing $-1$ is attached along each monodromy curve in the factorization, which can be visualized on the fiber $F_k$. 

By Theorem~\ref{thm: nw}, any \emph{minimal} weak symplectic filling $W$  of $(ST^{*}N_{k}, \xi_{can})$ is diffeomorphic to a Lefschetz fibration over the disk which corresponds to some positive factorization of $\phi_k$. Moreover, by Lemma~\ref{lem: monocurve} and Remark~\ref{rem: k2} (2), each positive factorization of $\phi_k$ has $k+2$ Dehn twists. Thus, the total space of the correponding Lefschetz fibration has a handle decomposition consisting of a $0$-handle, $2k+1$ $1$-handles and $k+2$ $2$-handles.

Suppose $k \geq 3$. Then by Lemma~\ref{lem: monocurve},  for any $j \in\{0, 1, 2,\ldots, k+1\}$, the monodromy curve $V'_{j}$ which appears in a factorization of $\phi_k$ encloses  the same holes as $V_j$ on the planar surface $F_k$. It follows that the linking number of the circle  $V'_j$ with any dotted circle in the corresponding Kirby diagram,  is the same as the linking number of $V_j$ with that dotted circle.    Hence we deduce (see, for example,  \cite[page 42]{ozst})  that the homology groups of the total space of the Lefschetz fibration is independent of the positive factorization of $\phi_k$. As a consequence,  since we already know one positive factorization of $\phi_k$ which gives a Lefschetz fibration on the disk cotangent bundle  $DT^*N_k$ (see Remark~\ref{rem: kir}),  we conclude that
$$H_1(W; \mathbb{Z}) = H_1(DT^*N_k ; \mathbb{Z})= H_1 (N_k ; \mathbb{Z}) = \mathbb{Z}^{k-1} \oplus \mathbb{Z}_2,$$ and 
$$ H_2(W; \mathbb{Z}) = H_2(DT^*N_k ; \mathbb{Z})= H_2 (N_k ; \mathbb{Z})=0.$$  

For $k=2$, there are two possible configurations of monodromy curves, the standard one as above and another one as
described in Remark~\ref{rem: k2} (2). For the standard one, the proof above is valid. For the other one, one can simply check that the homology groups are the same as in the standard one.  
\end{proof}

\section{Homotopy type of the fillings}

Our goal in this section is to prove Theorem~\ref{thm: main} that we stated in the introduction. We begin with some basic observations.

The unit cotangent bundle $ST^{*}N_{k}$ is a circle bundle over the nonorientable surface 
$N_k$ whose Euler number is equal to $-\chi(N_k)=k-2$. The fundamental group of $ST^{*}N_{k}$ is given (cf. \cite[page 91]{j}) as follows:  
$$\pi_1(ST^{*}N_{k})=<a_1,\ldots, a_k, t \; | \; a_{j}ta_{j}^{-1}=t^{-1}, \prod\limits_{j=1}^{k} a_{j}^{2}=t^{k-2}>.$$
Here $t$ represents the homotopy class of the circle fiber and it generates a cyclic normal subgroup of $\pi_1(ST^{*}N_{k})$.   After abelianization,  we get
$$H_1(ST^{*}N_{k}; \mathbb{Z})=
\begin{cases}
\mathbb{Z}^{k-1}\oplus  \mathbb{Z}_{2}\oplus  \mathbb{Z}_{2}, & \text{if} \; k\text{ is even,}\\
\mathbb{Z}^{k-1}\oplus \mathbb{Z}_{4}, &\text{if} \; k\text{ is odd.}
\end{cases}
$$

\begin{proposition} \label{prop: asph} Any Stein  filling of  $(ST^{*}N_{k}, \xi_{can})$ is aspherical, provided that $k\geq 2$.  
\end{proposition} 

\begin{proof} Suppose that  
$W$ is a Stein filling of $(ST^{*}N_{k}, \xi_{can})$. 
Then, since $W$ has a handlebody consisting of  handles of index at most two,  the  inclusion map $$i: ST^{*}N_{k} = \partial W \rightarrow W$$
 induces a surjective homomorphism $$i_* : \pi_1(ST^{*}N_{k}) \rightarrow \pi_1(W)$$ and hence we obtain the following commutative diagram,
$$\xymatrix{
    \pi_1(ST^{*}N_{k}) \ar[rr]^{i_{*}}\ar[d]_{f_1} & & \pi_1(W) \ar[d]^{f_2} \\
     H_1(ST^{*}N_{k}; \mathbb{Z})\ar[rr]^{i_*}&  & H_1(W; \mathbb{Z})
    }$$
where $f_1$ and $f_2$ are Hurewicz maps.

Let $\alpha_{j}=i_{*}(a_j)$, for $j=1,\ldots, k$, and let $\tau=i_{*}(t)$. Then $f_{2}(\alpha_{j})$, $j=1,\ldots, k$, and $f_{2}(\tau)$ generate $H_1(W; \mathbb{Z})$. If $k$ is even, then both $\sum\limits_{j=1}^{k}f_{1}(a_j)$ and $f_{1}(t)$ are torsion of order $2$ in $H_1(ST^{*}N_{k}; \mathbb{Z})$. If $k$ is odd, then $\sum\limits_{j=1}^{k}f_{1}(a_j)$ is torsion of order 4 and $f_{1}(t)=2\sum\limits_{j=1}^{k}f_{1}(a_j)$  in $H_1(ST^{*}N_{k}; \mathbb{Z})$.  We know that  $H_1(W; \mathbb{Z})=\mathbb{Z}^{k-1} \oplus \mathbb{Z}_2$, by Proposition~\ref{prop: hom}. Thus $\sum\limits_{j=1}^{k}f_{2}(\alpha_j)\in H_1(W; \mathbb{Z})$  is either trivial or   torsion of order $2$ if $k$ is even,  and is torsion of order $2$ if $k$ is odd.  Moreover,  $f_{2}(\alpha_j)$ is non-torsion in $H_1(W; \mathbb{Z})$  for $j=1,\ldots, k$.  Therefore, we can define a surjective homomorphism 
$\psi : H_1(W; \mathbb{Z})\rightarrow \mathbb{Z}_{2}$ such that $\psi(f_{2}(\alpha_j))=1$ for all $j =1, \ldots, k$ and $\psi(f_{2}(\tau))=0$.

Now consider the  double covers of $ST^{*}N_{k}$ and $W$ induced by  
the surjective homomorphisms $$\psi\circ f_2\circ i_{*}:\pi_1(ST^{*}N_{k})\rightarrow \mathbb{Z}_2 \:\; \text{and} \:\; \psi\circ f_2:\pi_1(W)\rightarrow \mathbb{Z}_2, $$ respectively. Since each generator $a_j$, for $j=1,\ldots, k$,  is mapped to $1\in \mathbb{Z}_2$, and $t$ is mapped to $0\in \mathbb{Z}_2$, the double cover of $ST^{*}N_{k}$ is the unit cotangent bundle $ST^{*}\Sigma_{k-1}$, where  $\Sigma_{k-1}$ is  the closed orientable surface  which is the double cover of the nonorientable surface $N_k$.

By \cite{g1}, the canonical contact structure on $ST^{*}N_{k}$ lifts to the canonical contact structure on $ST^{*}\Sigma_{k-1}$ under this double cover.  On the other hand, since any finite cover of a Stein surface is Stein (cf. \cite[page 436]{gs}), the corresponding double cover $W'$ of $W$  is a Stein filling of $(ST^{*}\Sigma_{k-1}, \xi_{can})$. By \cite[Proposition 4.9]{sm}, $W'$ is aspherical, i.e., $\pi_n (W')=0$ 
for all $n\geq 2$.   Hence $W$ is also aspherical, since it is well-known (see, for example, \cite[Proposition 4.1]{h}) that higher homotopy groups are preserved under coverings.  
\end{proof}

Next we compute the fundamental group of the  fillings.

\begin{proposition} \label{prop: fund}  If  $W$ is a Stein filling of  $(ST^{*}N_{k}, \xi_{can})$, then $\pi_1(W)=\pi_1 (N_k)$, provided that $k \geq 2$. 
\end{proposition} 

\begin{proof}
Suppose that $W$ is a Stein 
filling of  $(ST^{*}N_{k}, \xi_{can})$. Since $\langle t \rangle $ is a normal subgroup in $\pi_{1}(ST^{*}N_{k})$,  its image $\langle \tau \rangle $ under the surjective homomorphism $i_{*}:\pi_{1}(ST^{*}N_{k})\rightarrow \pi_{1}(W)$ is a 
normal subgroup of $\pi_{1}(W)$.  Note that $i_*$  induces a surjective homomorphism 
$$p:  \pi_1 (N_k) \rightarrow \pi_{1}(W)/ \langle \tau \rangle.$$  

Let  $W'$ be the double cover of $W$ as in the proof of Proposition~\ref{prop: asph}. Then the surjective homomorphism 
$i_{*}:\pi_{1}(ST^{*}\Sigma_{k-1})\rightarrow \pi_{1}(W')$ induces a surjective homomorphism 
$$p':\pi_{1}(\Sigma_{k-1})\rightarrow \pi_{1}(W')/\langle \tau' \rangle , $$ where $\tau'= i_{*}(t')$, and $t'$ represents the homotopy class of the circle fiber of $ST^{*}\Sigma_{k-1}$.

Using the same argument as in the proof of \cite[Proposition 4.7]{sm}, we have $\ker(p)=\ker(p')$, where  
$\pi_{1}(\Sigma_{k-1})$ is identified as a subgroup of $\pi_1 (N_k)$. Now, since $W'$ is a Stein filling of $(ST^{*}\Sigma_{k-1}, \xi_{can})$, we conclude that the surjective homomorphism $p'$ above is also injective by \cite[Proposition 4.8]{sm}. Hence we get  $\ker(p)=\ker(p')=0$.  As a consequence,  we have a short exact sequence $$ 1\rightarrow \langle\tau\rangle\rightarrow \pi_1(W)\rightarrow \pi_1 (N_k)\rightarrow 1.$$

Since $\langle \tau \rangle$ is a cyclic subgroup of $\pi_{1}(W)$,  isomorphic to $\mathbb{Z}_{m}$ for 
some non-negative integer $m$, where $\mathbb{Z}_{0} = \mathbb{Z}$, it follows that the short exact sequence above can be expressed as  
$$1\rightarrow \mathbb{Z}_{m}\rightarrow \pi_1(W)\rightarrow \pi_1 (N_k)\rightarrow 1.$$ Our goal is to show that $m=1$. In order to achieve our goal, we first observe that  $H_{2}(\pi_{1}(W); \mathbb{Z})=H_{2}(W; \mathbb{Z})$, since $W$ is aspherical by Proposition~\ref{prop: asph}. But since we have $H_{2}(W; \mathbb{Z})= 0$ by  Proposition~\ref{prop: hom}, we conclude that  $H_{2}(\pi_{1}(W); \mathbb{Z})=0$. 

Next we use the Lyndon/Hochschild-Serre spectral sequence (see, for example,  \cite{b})  
$$E_{p,q}^{2}=H_{p}(\pi_{1}(N_k); H_{q}(\mathbb{Z}_m; \mathbb{Z}))\Longrightarrow H_{p+q}(\pi_{1}(W); \mathbb{Z})$$ to compute   $H_{2}(\pi_{1}(W); \mathbb{Z})$, which a priori may depend on $m$. 

Based on the facts that $H_{0}(\mathbb{Z}_{m}; \mathbb{Z})= \mathbb{Z}$, $H_{1}(\mathbb{Z}_{m}; \mathbb{Z})=\mathbb{Z}_{m}$, 
  $H_{2}(\mathbb{Z}_{m}; \mathbb{Z})=0$, for all $ m \geq 0$, and that the nonorientable surface $N_k$ is aspherical for $k \geq 2$, we obtain 
$$E_{0,2}^{2}=H_{0}(\pi_{1}(N_k); H_{2}(\mathbb{Z}_{m}; \mathbb{Z}))=H_{0}(N_k; 0)=0,$$

$$E_{2,0}^{2}=H_{2}(\pi_{1}(N_k); H_{0}(\mathbb{Z}_{m}; \mathbb{Z}))=H_{2}(N_k; \mathbb{Z})=0,$$ and

$$E_{1,1}^{2}=H_{1}(\pi_{1}(N_k); H_{1}(\mathbb{Z}_{m}; \mathbb{Z}))=H_{1}(N_k; \mathbb{Z}_{m})=\begin{cases}
(\mathbb{Z}_{m})^{k-1}\oplus  \mathbb{Z}_{2}, & m\text{ is even,}\\
(\mathbb{Z}_{m})^{k-1}, & m\text{ is odd.}
\end{cases}$$

Since the cohomological dimension of the surface group $\pi_{1}(N_k)$  is equal to $2$ for $k >1$ (cf. \cite[Page 185]{b}), $E^{2}$ page is supported by  $ p\in\{0,1, 2\}$.  It follows that  

$$E_{0,2}^{\infty}=E_{0,2}^{2}, \; E_{1,1}^{\infty}=E_{1,1}^{2}, \; \mbox{and} \; E_{2,0}^{\infty}=E_{2,0}^{2}.$$

Hence we have  $$H_{2}(\pi_{1}(W); \mathbb{Z})=E_{0,2}^{\infty}\oplus E_{1,1}^{\infty}\oplus E_{2,0}^{\infty}=\begin{cases}
(\mathbb{Z}_{m})^{k-1}\oplus  \mathbb{Z}_{2}, & m\text{ is even,}\\
(\mathbb{Z}_{m})^{k-1}, & m\text{ is odd.}
\end{cases}$$  Therefore $m=1$, since $H_{2}(\pi_{1}(W); \mathbb{Z})=0$, and thus $\pi_{1}(W)=\pi_{1}(N_k)$. 
\end{proof}

We are now ready to prove our main result which is Theorem~\ref{thm: main}.

\begin{proof}[Proof of Theorem~\ref{thm: main}] Suppose that $k \geq 2$ and $W$ is a 
Stein filling of the  contact $3$-manifold  $(ST^{*}N_{k}, \xi_{can})$. We will show that $W$ is 
s-cobordant rel boundary to the disk cotangent bundle $DT^*N_k$, using the same argument as in the proof of \cite[Theorem 4.10]{sm}. For the sake of completeness, we outline the argument here.  We showed that $W$ is aspherical by Proposition~\ref{prop: asph},  and 
$\pi_1 (W) = \pi_1(N_k)$, a surface group, by Proposition~\ref{prop: fund}.  
According to   \cite[Corollary 1.23]{k},  such a compact manifold $W$ is topologically s-rigid. This condition implies 
that it suffices to find a homotopy equivalence 
$\rho :  DT^*N_k  \to W$ which restricts to a homeomorphism $ ST^*N_k  \to \partial W$ in order to prove that  $DT^*N_k$ is s-cobordant to $W$. 

Now consider  the standard handlebody decomposition of $DT^*N_k$ consisting of a $0$-handle, $k$ $1$-handles and a $2$-handle. By turning it upside down, we can construct $DT^*N_k$ by attaching a $2$-handle, $k$ $3$-handles and a $4$-handle to a 
thickened $ST^*N_k$. Next we define  a homeomorphism $\rho : ST^*N_k \to \partial W$ sending a circle fiber to a circle fiber. Note that the attaching curve of the (upside down) $2$-handle is a circle fiber in $\partial W$ which is nullhomotopic in $W$. Therefore,  $\rho$ extends over the $2$-handle of $DT^*N_k$. Moreover, $\rho$ extends over the handles  of index greater than $2$, since $W$ is aspherical. It follows that $\rho :  DT^*N_k  \to W$ is a homotopy equivalence by Whitehead's theroem.  

Finally, we use Theorem~\ref{thm: nw} to finish the proof for an arbitrary minimal weak symplectic filling of $(ST^{*}N_{k}, \xi_{can})$.
 \end{proof} 

\noindent{\bf{Acknowledgement:}} This work was mainly carried out at the Department of Mathematics at UCLA.  
 The authors would like to express their gratitude to Ko Honda for his hospitality.

\end{document}